\documentclass{amsart}
\usepackage{amsmath,amsfonts,amssymb,amsthm,url,verbatim}
\usepackage[arrow,matrix]{xy}

\newtheorem{tm}{Theorem}

\newtheorem{lemma}[tm]{Lemma}
\newtheorem{corollary}[tm]{Corollary}
\theoremstyle{definition}
\newtheorem*{definition}{Definition}
\theoremstyle{remark}
\newtheorem*{acknowledgments}{Acknowledgments}
\newtheorem*{notation}{Notation}

\newcommand{\Q}{\mathbb{Q}}

\newcommand{\C}{\mathbb{C}}
\newcommand{\Z}{\mathbb{Z}}

\newcommand{\F}{\mathbb{F}}
\newcommand{\PP}{\mathbb{P}}

\newcommand{\cC}{\mathcal{C}}
\newcommand{\cJ}{\mathcal{J}}
\newcommand{\OO}{\mathcal{O}}
\newcommand{\cX}{\mathcal{X}}

\newcommand{\fp}{\mathfrak{p}}

\newcommand{\fq}{\mathfrak{q}}

\newcommand{\X}{\mathrm{X}}
\newcommand{\Y}{\mathrm{Y}}

\newcommand{\tors}{\mathrm{tors}}

\newcommand\smallmat[4]{\bigl(\begin{smallmatrix} #1 & #2 \\
  #3 & #4 \end{smallmatrix}\bigr)}

\DeclareMathOperator{\Gal}{Gal}
\DeclareMathOperator{\GL}{GL}
\DeclareMathOperator{\id}{id}
\DeclareMathOperator{\Jac}{Jac}
\DeclareMathOperator{\lcm}{lcm}
\DeclareMathOperator{\SL}{SL}

\DeclareMathOperator{\Pic}{Pic}

\DeclareMathOperator{\rk}{rk}

\DeclareMathOperator{\Norm}{Nm}

\begin{document}

\title[A criterion to rule out torsion groups for elliptic curves]{A criterion to rule out torsion groups for elliptic curves over number fields}

\author{Peter Bruin}
\address{Mathematisch Instituut\\ Universiteit Leiden\\ Postbus 9512\\ 2300 RA \ Leiden\\ Netherlands}
\email{P.J.Bruin@math.leidenuniv.nl}
\author{Filip Najman}
\address{Department of Mathematics\\ University of Zagreb\\ Bijeni\v cka cesta 30\\ 10000 Zagreb\\ Croatia}
\email{fnajman@math.hr}
\address{Department of Mathematics, Massachusetts Institute of Technology, Cambridge, Massachusetts 02139, USA}
\email{fnajman@mit.edu}

\begin{abstract}
We present a criterion for proving that certain groups of the form $\Z/m\Z\oplus\Z/n\Z$ do not occur as the torsion subgroup of any elliptic curve over suitable (families of) number fields. We apply this criterion to eliminate certain groups as torsion groups of elliptic curves over cubic and quartic fields. We also use this criterion to give the list of all torsion groups of elliptic curves occurring over a specific cubic field and over a specific quartic field.
\end{abstract}

\maketitle

\section{Introduction}

A fundamental result in the theory of elliptic curves, the Mordell--Weil theorem, states that the Abelian group of points of an elliptic curve (or more generally an Abelian variety) $E$ over a number field $K$ is finitely generated. Thus, $E(K)$ is isomorphic to $E(K)_\tors\oplus\Z^r$, where $E(K)_\tors$ is the torsion subgroup of $E(K)$ and $r\geq 0$ is the rank of $E(K)$.

We will denote by $\Phi(d)$, where $d$ is a positive integer, the set of all the possible isomorphism types of $E(K)_\tors$, where $K$ runs through all number fields $K$ of degree~$d$ and $E$ runs through all elliptic curves over~$K$.  Similarly, for a fixed number field~$K$, we will denote by $\Phi(K)$ the set of all the possible isomorphism types of $E(K)_\tors$ where $E$ runs through all elliptic curves over this fixed field $K$. Obviously, if $K$ is a number field of degree~$d$, then $\Phi(K)\subseteq \Phi(d)$, and $\Phi(d)$ is the union of the $\Phi(K)$ with $K$ running over all number fields of degree~$d$.
It is interesting to determine the set $\Phi(d)$ for fixed integers~$d$, as well as the set $\Phi(K)\subseteq\Phi(d)$ for fixed number fields~$K$ of degree~$d$.

Let $C_m$ be a cyclic group of order $m$. Mazur's torsion theorem \cite{maz1,maz2} tells us that $\Phi(\Q) = \Phi(1)$ consists of the following 15 groups:
\begin{equation}
\label{rac}
\begin{aligned}
&C_m,  &  m&=1,\ldots, 10, 12,\\
& C_2 \oplus C_{2m}, & m&=1,\ldots, 4.\\
\end{aligned}
\end{equation}
Similarly, by a theorem of Kamienny \cite[Theorem~3.1]{Kam1} and Kenku and Momose \cite[Theorem~(0.1)]{km}, $\Phi(2)$ consists of the following 26 groups:
\begin{equation}
\label{kvad}
\begin{aligned}
&C_m,  &  m&=1,\ldots, 16, 18,\\
& C_2 \oplus C_{2m}, & m&=1,\ldots, 6,\\
& C_3 \oplus C_{3m}, &  m&=1,2,\\
&C_4 \oplus C_4.
\end{aligned}
\end{equation}
Over cubic fields, we know that if a point on an elliptic curve has prime order $p$, then $p\leq 13$ \cite{Par1,Par2}, we know all the isomorphism types in $\Phi(3)$ that appear as the torsion groups of infinitely many non-isomorphic elliptic curves \cite{jks}, and we know that this list is strictly smaller than $\Phi(3)$ \cite{naj4}, as opposed to what happens in the rational and quadratic cases.
Unfortunately, $\Phi(d)$ has as yet not been determined for any $d\geq 3$, although Maarten Derickx (personal communication) has told us that the computation of $\Phi(3)$ should be within reach.

As for results on $\Phi(K)$ for specific $K$, Mazur \cite{maz1} determined $\Phi(\Q) = \Phi(1)$, the second author determined $\Phi(\Q(\zeta_3))$ and $\Phi(\Q(\zeta_4))$ \cite{naj0}, and methods of determining $\Phi(K)$ for other quadratic fields $K$ were given by Kamienny and the second author \cite{kn}.
The second author also tried to determine $\Phi(K)$ for certain cubic fields $K$ with small discriminant, but managed to obtain only partial results \cite{naj3}.


In this paper we develop a criterion, based on a careful study of the cusps of modular curves $X_1(m,n)$, which can tell us that certain groups do not occur as torsion groups of elliptic curves over a number field $K$.
This criterion is essentially a generalization of a criterion of Kamienny \cite{Kam3}.
Kamienny showed that for certain $n$, the curve $X_1(n)$ cannot have non-cuspidal points over an extension of degree~$d$ of~$\Q$, where $d$ is less than the gonality of $X_1(n)$, as  points of degree~$d$ on $X_1(n)$ would force functions of a smaller degree than the gonality to exist, which is impossible. 

We generalize Kamienny's approach both by using the modular curves $X_1(m,n)$ instead of $X_1(n)$ and by viewing the number fields $K$ as extensions of a suitable subfield~$L$ of $\Q(\zeta_m)$.
This generalization is somewhat technical (for example, it requires a careful consideration of the fields of definitions of cusps), but gives us more flexibility in ruling out torsion groups of the form $C_m \oplus C_n$.
Our criterion, on its own or in combination with other techniques, allows us to advance our understanding of the torsion groups of elliptic curves over~$K$, both when $K$ is a fixed number field and when $K$ runs through all number fields of degree $d$.
In particular, we make progress in determining $\Phi(3)$ and $\Phi(4)$ by ruling out a number of possibilities for torsion groups of elliptic curves over cubic and quartic fields.
As for determining $\Phi(K)$ for a fixed cubic or quartic field~$K$, a natural choice for a quartic field~$K$, in view of \cite{naj0}, is the `next' cyclotomic field, $\Q(\zeta_5)$.
Since there are no cubic cyclotomic fields, we choose the cyclic cubic field $\Q(\zeta_{13}+\zeta_{13}^5+\zeta_{13}^8+\zeta_{13}^{12})$.

In Section~\ref{sec:thm}, we state and prove our main results (Theorem~\ref{thm:main} and Corollary~\ref{cor}).
In Sections \ref{sec:cubic13} and~\ref{sec:quartic5}, we use Theorem~\ref{thm:main} to determine $\Phi(\Q(\zeta_{13}+\zeta_{13}^5+\zeta_{13}^8+\zeta_{13}^{12}))$ and $\Phi(\Q(\zeta_5))$, achieving to our knowledge the first determination of $\Phi(K)$ for a cubic and a quartic field, respectively.
In Sections \ref{sec:cubicgeneral} and~\ref{sec:quarticgeneral}, we use Corollary~\ref{cor}, together with other techniques, to prove that a large number of finite groups do not occur as torsion groups of elliptic curves over cubic and quartic fields, respectively.

The computer calculations were done using Magma \cite{mag}. Showing that the rank of Jacobians over $\Q$ is 0, unless otherwise mentioned, has been done by showing that the $L$-function of (the factors of) $J$ is non-zero.
By results of Kato \cite{kato}, the Birch-Swinnerton--Dyer conjecture is true for quotients of modular Jacobians, so this computation unconditionally proves that the rank is 0.

\section{Main results}

\label{sec:thm}

\begin{notation}
If $K$ is a number field, $\OO_K$ denotes its ring of integers.  If $\fp$ is a prime ideal of~$\OO_K$, we write $k(\fp)$ for the residue field $\OO_K/\fp$, and $\Norm(\fp)=\#k(\fp)$.  Furthermore, we denote by $\OO_{K,\fp}$ the localization of~$\OO_K$ at~$\fp$.
\end{notation}

\begin{definition}
Let $m$ and~$n$ be positive integers with $m\mid n$.
Let $L$ be a subfield of $\Q(\zeta_m)$, and let $\fp_0$ be a prime of~$L$.
Let $X$ be the curve $X_1(m,n)_{\Q(\zeta_m)}$ viewed as a (proper, smooth, but possibly geometrically disconnected) curve over~$L$.
We consider triples $(\cX,\cX',\pi)$, where
\begin{itemize}
\item $\cX$ is a flat, proper model of~$X$ over $\OO_{L,\fp_0}$ such that the $j$-invariant extends to a map $j\colon\cX\to\PP^1_{\OO_{L,\fp_0}}$;
\item $\cX'$ is a flat, proper and regular curve over $\OO_{L,\fp_0}$
such that the curve $X'=\cX'_L$ over~$L$ is geometrically connected;
\item $\pi\colon\cX\to\cX'$ is a proper and generically finite map of $\OO_{L,\fp_0}$-schemes.
\end{itemize}
Given such a triple $(\cX,\cX',\pi)$, we write $\cC$ for the topological inverse image of the section~$\infty$ under the map $j\colon\cX\to\PP^1_{\OO_{L,\fp_0}}$, and $\cC'$ for the topological image of~$\cC$ under~$\pi$ equipped with the reduced induced subscheme structure.
With this notation, we call $(\cX,\cX',\pi)$ a \emph{nice $(L,\fp_0)$-quotient} of $X_1(m,n)$ if the following conditions are satisfied:
\begin{enumerate}
\item The scheme $\cC'$ is normal and lies in the smooth locus of\/~$\cX'$ over~$\OO_{L,\fp_0}$.
\item The image of the open subscheme $\cX\setminus\cC$ under~$\pi$ equals $\cX'\setminus\cC'$.
\end{enumerate}
\end{definition}


\begin{tm}
\label{thm:main}
Let $A$ be a group of the form $C_m\oplus C_n$ with $m\mid n$, let $K$ and~$L$ be number fields with $L\subseteq \Q(\zeta_m)\subseteq K$, and let $d=[K:L]$.  Let $\fp_0$ be a prime of~$L$, let $p$ be the residue characteristic of\/~$\fp_0$, and let $e$ be the largest absolute ramification index of a prime of~$K$ dividing $\fp_0$.  Let $S_{K,\fp_0}$ be the set
$$
S_{K,\fp_0} = \{\delta\ge1\mid \delta\text{ divides $[k(\fp):k(\fp_0)]$ for some prime $\fp$ of~$K$ over~$\fp_0$}\}.
$$
Let
$$
A'=\begin{cases}
  A & \text{ if }p>e+1\\
  \text{maximal $p$-divisible subgroup of }A & \text{ if }p\leq e+1.
\end{cases}
$$
Let $(\cX,\cX',\pi)$ be a nice $(L,\fp_0)$-quotient of $X_1(m,n)$.
Let $X'=\cX'_L$, and let $J'$ be the Jacobian of~$X'$.
Let $h$ be the least common multiple of the ramification indices $e(\fq/\fp_0)$ where $L(C)$ is the function field of an irreducible component~$C$ of\/~$\cC'$ and $\fq$ is a prime of $L(C)$ over~$\fp_0$.
Assume that the following conditions are satisfied:
\begin{itemize}
\item[i)] The gonality of $X'$ over~$L$ is at least $dh+1$.
\item[ii)] The group $J'(L)$ has rank $0$.
\item[iii)] If $p=2$, then the $2$-torsion subgroup of $J'(L)$ is trivial.
\item[iv)] For all primes $\fp\mid\fp_0$ of~$K$, there does not exist an elliptic curve over~$k(\fp)$ with a subgroup isomorphic to $A'$.
\item[v)] For all primes $\fp\mid\fp_0$ of~$K$, neither $3\Norm(\fp)$ nor $4\Norm(\fp)$ is divisible by~$\#A'$.
\item[vi)] For all irreducible components $C$ of\/~$\cC'$, if the function field $L(C)$ has a prime $\fq$ over~$\fp_0$ such that $[k(\fq):k(\fp_0)]$ is in $S_{K,\fp_0}$, then $\fq$ is the \emph{unique} prime of~$L(C)$ over~$\fp_0$.
\end{itemize}
Then there does not exist an elliptic curve over~$K$ with a subgroup isomorphic to~$A$.
\end{tm}

We will prove Theorem~\ref{thm:main} below; we begin with an auxiliary result.

\begin{lemma}
\label{lemma:mult}
Let $A$, $K$, $L$ and $\fp_0$ be as in Theorem~\ref{thm:main}.  Under the conditions iv) and~v) of Theorem~\ref{thm:main}, any elliptic curve $E$ over~$K$ equipped with an embedding $\iota\colon A\rightarrowtail E(K)$ has multiplicative reduction at all primes of~$K$ lying over~$\fp_0$.
\end{lemma}

\begin{proof}
Let $\fp$ be a prime of~$K$ over~$\fp_0$.  By $\tilde E_\fp$ we denote the reduction of~$E$ modulo~$\fp$, i.e.\ the special fibre of the N\'eron model of~$E$ at~$\fp$.  Then we have a reduction map
$$
E(K)\to\tilde E_\fp(k(\fp)).
$$
This map is injective on~$\iota(A')$ by \cite[Appendix]{kat} and the definition of~$A'$.  The group $\tilde E_\fp(k(\fp))$ therefore contains a subgroup isomorphic to~$A'$.

By assumption iv), $E$ does not have good reduction at~$\fp$. If $E$ had additive reduction, then by the Kodaira--N\'eron classification \cite[Appendix~C, \S15]{sil}, $\tilde E_\fp(k(\fp))$ would be a product of the additive group of~$k(\fp)$ and a group of order $\leq 4$, contradicting assumption v).  We conclude that $E$ has multiplicative reduction at~$\fp$.
\end{proof}

\begin{proof}[Proof of Theorem~\ref{thm:main}]
Let $\cX_{\fp_0}$ and $\cX'_{\fp_0}$ be the special fibres of $\cX$ and $\cX'$ over~$\fp_0$.
Let $\cJ'$ be the N\'eron model of~$J'$ over~$\OO_{L,\fp_0}$.
It is known \cite[\S9.5, Theorem~4]{blr} that $\cJ'$ represents the functor $P/E$, where $P$ is the open subfunctor of $\Pic_{\cX'/\OO_{L,\fp_0}}$ given by line bundles of total degree~$0$ and $E$ is the schematic closure in~$P$ of the unit section in $P(L)$.
We have a commutative diagram
$$
\xymatrix{
& P(\OO_{L,\fp_0}) \ar[r] \ar[d] & P(k(\fp_0)) \ar[d] \\
J'(L) & \cJ'(\OO_{L,\fp_0}) \ar[r] \ar@{=}[l] & \cJ'(k(\fp_0))\rlap.}
$$
By assumptions ii) and iii) and \cite[Appendix]{kat}, the bottom horizontal map is injective.  

Suppose the theorem is false.
Let $E$ be an elliptic curve over~$K$ equipped with an embedding $A\rightarrowtail E(K)$.
These data determine a point of $X(K)$ whose Zariski closure is a prime divisor $D$ on~$\cX$.
Let $D_{\fp_0}$ be the schematic intersection of the divisor~$D$ with~$\cX_{\fp_0}$, and let $(\pi_* D)_{\fp_0}$ be the schematic intersection of $\pi_* D$ with~$\cX'_{\fp_0}$.
By Lemma~\ref{lemma:mult}, $E$ has multiplicative reduction at all primes of~$K$ over~$\fp_0$, so the support of~$D_{\fp_0}$ is contained in~$\cC$.
Let $Z$ be the support of $(\pi_* D)_{\fp_0}$; then the definition of~$\cC'$ implies that $Z$ is contained in~$\cC'$.
We can write $(\pi_* D)_{\fp_0}$ as a linear combination $\sum_{z\in Z} n_z z$, where the $n_z$ are positive integers.

Let $z$ be a point of~$Z$.
Since $(\cX,\cX',\pi)$ is a nice $(L,\fp_0)$-quotient, there is a unique irreducible component~$C_z$ of~$\cC'$ containing $z$, and the coordinate ring of~$C_z$ is the integral closure of $\OO_{L,\fp_0}$ in the function field $L(C_z)$ of~$C_z$.
Hence $k(z)$ can be identified with the residue field $k(\fq_z)$ of some prime~$\fq_z$ of $L(C_z)$ over~$\fp_0$; in particular, $[k(\fq_z):k(\fp_0)]$ equals $[k(z):k(\fp_0)]$.
On the other hand, $k(z)$ can also be identified with a subfield of the residue field $k(\fp_z)$ of some prime~$\fp_z$ of~$K$ over~$\fp_0$, so $[k(z):k(\fp_0)]$ divides $[k(\fp_z):k(\fp_0)]$.
This implies that $[k(\fq_z):k(\fp_0)]=[k(z):k(\fp_0)]$ is in $S_{K,\fp_0}$.
By assumption~vi), $\fq_z$ is the only prime of $L(C_z)$ over~$\fp_0$.
This implies that the schematic intersection of $C_z$ with~$\cX'_{\fp_0}$ equals $e_z z$, where $e_z=e(\fq_z/\fp_0)$.
We note that $e_z$ divides $h$.

We consider the effective divisor $D'$ on~$\cX'$ defined by
$$
D'=\sum_z \frac{n_z h}{e_z} C_z.
$$
By the above description of the intersections of~$D$ and the $C_z$ with~$\cX'_{\fp_0}$, the divisor $h\pi_* D - D'$ on~$\cX'$ specializes to the zero divisor on~$\cX'_{\fp_0}$.
This implies that the class of $h\pi_* D - D'$ in $P(k(\fp_0))$ is zero.
By the commutativity of the above diagram and the injectivity of the bottom map, the class of $h\pi_* D - D'$ in $J'(L)$ is also zero.
By assumption i), we conclude that the divisors $h\pi_* D$ and~$D'$ are equal.
The generic fibre of $D$ is supported outside $\cC$; since $(\cX,\cX',\pi)$ is a nice $(L,\fp_0)$-quotient, the generic fibre of~$\pi_* D$ is supported outside $\cC'$.
On the other hand, the generic fibre of~$D'$ is supported on $\cC'$, a contradiction.
\end{proof}

The following corollary of Theorem \ref{thm:main} is useful for eliminating groups from $\Phi(d)$.

\begin{corollary}
\label{cor}
Let $A$ be a group of the form $C_m\oplus C_n$ with $m\mid n$, let $d\geq 1$ be an integer, and let $L$ be a subfield of\/ $\Q(\zeta_m)$. Let $\fp_0$ be a prime of~$L$, let $p$ be the residue characteristic of~$\fp_0$, and let $q=\Norm(\fp_0)$.
Let
$$
S_{\fp_0} = \left\{\delta\ge1\biggm|
\vcenter{\hbox{$K$ is an extension of\/ $\Q(\zeta_m)$ with $[K:L]=d$,}
\hbox{$\fp$ is a prime of $K$ over $\fp_0$, and $\delta$ divides $[k(\fp):k(\fp_0)]$}}
\right\}.
$$
Let
$$
A'=\begin{cases}
  A & \text{ if }p>d+1\\
  \text{maximal $p$-divisible subgroup of }A & \text{ if }p\leq d+1.
\end{cases}
$$
Let $(\cX,\cX',\pi)$ be a nice $(L,\fp_0)$-quotient of $X_1(m,n)$.
Let $X'=\cX'_L$, and let $J'$ be the Jacobian of~$X'$.
Let $h$ be the least common multiple of the ramification indices $e(\fq/\fp_0)$ where $L(C)$ is the function field of an irreducible component~$C$ of\/~$\cC'$ and $\fq$ is a prime of $L(C)$ over~$\fp_0$.
Assume that the following conditions are satisfied:
\begin{itemize}
\item[i)] The gonality of $X'$ over~$L$ is at least $dh+1$.
\item[ii)] The group $J'(L)$ has rank~$0$.
\item[iii)] If $p=2$, then the $2$-torsion subgroup of $J'(L)$ is trivial.
\item[iv)] For all $i\in S_{\fp_0}$, there does not exist an elliptic curve over a field of $q^i$ elements with a subgroup isomorphic to $A'$.
\item[v)] Neither $3q^d$ nor $4q^d$ is divisible by~$\#A'$.
\item[vi)] For all irreducible components $C$ of\/~$\cC'$, if the function field $L(C)$ has a prime $\fq$ over~$\fp_0$ such that $[k(\fq):k(\fp_0)]$ is in~$S_{\fp_0}$, then $\fq$ is the \emph{unique} prime of~$L(C)$ over~$\fp_0$.
\end{itemize}
Then there does not exist an elliptic curve over an extension of degree~$d$ of\/ $L$ with a subgroup isomorphic to~$A$.
\end{corollary}

\begin{proof}
Under the conditions of the corollary, the conditions of Theorem~\ref{thm:main} are satisfied for every extension~$K$ of degree~$d$ over $L$ such that $L\subseteq \Q(\zeta_m) \subseteq K$.
\end{proof}

We end this section with some remarks on checking the conditions of Theorem~\ref{thm:main} and Corollary~\ref{cor}.
The conditions are straightforward to check in practice, apart from condition ii) if $L\neq \Q$. 	
An easy way to make sure that condition vi) holds is to choose $\fp_0$ totally inert in $\Q(\zeta_n)$.

An important special case occurs when $L$ equals $\Q(\zeta_m)$, the prime $\fp_0$ does not divide~$n$, the curve $X'$ equals $X$ and $\pi$ is the identity on~$X$.
In this case the conditions simplify as follows: $(X,X',\pi)$ automatically extends to a nice $(L,\fp_0)$-quotient, and we have $A'=A$ and $h=1$.
Moreover, the following remarks are useful to check condition vi) in these cases.

Let $r$ be a divisor of~$n$.  The cusps of $X_1(m,n)$ represented by points $(a:b)\in\PP^1(\Q)$, where $a$, $b$ are coprime integers with $\gcd(b,n)=r$, all have the same field of definition, which we denote by $F_{m,n,r}$.  
By generalities on cusps and by the existence of the Weil pairing, we have $\Q(\zeta_m) \subseteq F_{m,n,r} \subseteq \Q(\zeta_n)$.
Explicitly, the field $F_{m,n,r}$ can be described as follows.
We consider the subgroups $H_{m,n,r}^0\subseteq H_{m,n,r}\subseteq G_{m,n}\subseteq(\Z/n\Z)^\times$ defined by
$$
\begin{aligned}
G_{m,n} &= \{s\in(\Z/n\Z)^\times\mid s\equiv 1 \pmod m\},\\
H_{m,n,r}^0 &= \{s\in(\Z/n\Z)^\times\mid s\equiv 1 \pmod{\lcm(m,n/r)}\},\\
H_{m,n,r} &= \begin{cases}
H_{m,n,r}^0 & \text{if } \gcd(mr,n) > 2,\\
H_{m,n,r}^0 \cdot \{\pm1\} & \text{if } \gcd(mr,n)\le 2.
\end{cases}
\end{aligned}
$$
(Note that in the latter case $m$ is 1 or~2, so $-1$ is in $G_{m,n}$.)
Using the canonical identification of $\Gal(\Q(\zeta_n)/\Q)$ with $(\Z/n\Z)^\times$, the field $F_{m,n,r}$ is then the field of invariants of $H_{m,n,r}$ acting on~$\Q(\zeta_n)$.
In the case $m=1$, we have
$$
F_{1,n,r}=\begin{cases}
\Q(\zeta_{n/r})   & \text{if } r > 2,\\
\Q(\zeta_{n/r})^+ & \text{if } r\le 2.\\
\end{cases}
$$
Condition vi) can be checked by computing a defining polynomial for $F_{m,n,r}$ over~$L$ and factoring it modulo~$\fp_0$.

\section{Torsion groups of elliptic curves over $\Q(\zeta_{13}+\zeta_{13}^5+\zeta_{13}^8+\zeta_{13}^{12})$}

\label{sec:cubic13}

In this section, we consider the cyclic cubic field
$$
K=\Q(\zeta_{13}+\zeta_{13}^5+\zeta_{13}^8+\zeta_{13}^{12})
$$
of discriminant~$13^2$.  This field can be written as $K=\Q(\omega)$ where $\omega$ has minimal polynomial $x^3 + x^2 - 4x + 1$.

\begin{tm}
\label{tm:cubic}
For every elliptic curve $E$ over~$K$, the torsion group $E(K)_\tors$ is one of the groups from Mazur's torsion theorem.
\end{tm}

\begin{proof}
By results of Parent \cite{Par1,Par2}, we know that no prime $p>13$ divides the torsion order of an elliptic curve over a cubic field. It therefore remains to prove that $E(K)$ does not contain any of the following groups:
$$
\begin{aligned}
&C_n& \text{where }n &= 11,13,14,15,16,18,20,21,24,25,27,35,49,\\
&C_2 \oplus C_{2m}& \text{where } m &= 5,6.\\
\end{aligned}
$$
For later use, we note that $3$ is totally inert, $5$ is totally split and $13$ is totally ramified in~$K$.
Let $\fp_5$ be one of the primes above~$5$, and let $\fp_{13}$ be the unique prime above~$13$.

\subsection{The cases $C_{11}$, $C_{14}$ and $C_{15}$}

In these cases, the modular curve $X_1(n)$ is an elliptic curve, and an easy computation in Magma shows that $X_1(n)(K)=X_1(n)(\Q)$.  It is well known that $X_1(n)(\Q)$ contains only cusps (see for example \cite{maz1}), hence $Y_1(n)(K)=\emptyset$.

\subsection{The cases $C_{20}$, $C_{24}$, $C_2\oplus C_{10}$ and $C_2\oplus C_{12}$}

Recall that $X_0(20)$ and $X_0(24)$ are elliptic curves. A computation in Magma shows that $X_0(20)(K)=X_0(20)(\Q)$ and $X_0(24)(K)=X_0(24)(\Q)$, and it is known that $X_0(20)(\Q)$ and $X_0(24)(\Q)$ consist purely of cusps (see for example \cite{maz2}).
As an elliptic curve with a point of order~$n$ admits a cyclic isogeny of degree~$n$, it follows that $Y_1(20)(K)=\Y_1(24)(K)=\emptyset$. Similarly, an elliptic curve with torsion $C_2 \oplus C_{2m}$ over~$K$ is $2$-isogenous to a curve with a cyclic $2m$-isogeny, and hence it follows that $Y_1(2,10)(K)=\Y_1(2,12)(K)=\emptyset$.

\subsection{The cases $C_{13}$, $C_{16}$ and $C_{18}$}

The modular curves $X_1(n)$ are all hyperelliptic curves of genus~2. Magma computations show that $\rk J_1(13)(k)=\rk J_1(16)(K)=\rk J_1(18)(K)=0$.

We compute $J_1(13)(\F_5)\simeq C_{19}$, and as $J_1(13)(K)\rightarrowtail J_1(k(\fp_5))=J_1(\F_5)$ and $J_1(\Q)\simeq C_{19}$, it follows that $J_1(13)(K)=J_1(13)(\Q)\simeq C_{19}$. It follows that $Y_1(13)(K)=\emptyset$. A similar argument deals with the case $n=18$.

For $n=16$, it is not enough to use just one prime. We compute $\#J_1(16)(\F_5)=40$, $\#J_1(16)(\F_{27})=1220$  and $J_1(16)(\Q)\simeq C_2 \oplus C_{10}$. Since $\gcd(40,1220)=20$, it follows that $J_1(16)(K)=J_1(16)(\Q)$. We conclude that $Y_1(16)(K)=\emptyset$.

\subsection{The cases $C_n$ for $n=21$, $25$, $27$, $35$, $49$}

We apply Theorem~\ref{thm:main} using $L=\Q$, $\fp_0=(5)$ for $n=21,27,49$ and $\fp_0=(13)$ for $n=25,35$, $X'=X$ and $\pi=\id$.
Conditions i), ii) and iii) are satisfied since in all cases $X_1(n)$ is of gonality $\geq 4$ \cite{jks} and $\rk J_1(n)(\Q)=0$.
Condition iv) holds because of the Hasse bound over $\F_5$ and~$\F_{13}$.
Condition v) clearly holds.
For $n\in\{27,35,49\}$, condition vi) holds because $\fp_0$ is totally inert in $\Q(\zeta_n)$.
For $n=21$, condition vi) holds because $S_{K,(5)}=\{1\}$ and there are no primes of degree~$1$ above~$5$ in $\Q(\zeta_3)$, $\Q(\zeta_7)$ and $\Q(\zeta_{21})^+$.
Finally, for $n=35$, condition vi) holds because $S_{K,(13)}=\{1\}$ and there are no primes of degree~$1$ above~$13$ in $\Q(\zeta_5)$, $\Q(\zeta_7)$ and $\Q(\zeta_{35})^+$.

\medbreak\noindent
Theorem~\ref{tm:cubic} follows by combining the above cases.
\end{proof}

\section{Torsion groups of elliptic curves over $\Q(\zeta_{5})$}

\label{sec:quartic5}

In this section, we will determine $\Phi(K)$, where $K$ is the field $\Q(\zeta_5)$.
We will use the (as of yet unpublished) result that the largest prime dividing the order of a point over a quartic field is $17$ \cite{dkss}, and the fact that there is no $17$-torsion over cyclic quartic extensions of $\Q$ \cite{maz_new}.

\begin{tm}
\label{tm:quartic5}
Let $K=\Q(\zeta_{5})$. Then for every elliptic curve $E$ over~$K$, the torsion group $E(K)_\tors$ is one of the following groups:
\begin{equation}
\label{quart:list}
\begin{aligned}
&C_n & \text{ where }n&=1,\ldots, 10, 12, 15, 16,\\
&C_2\oplus C_{2m} & \text{ where }m&=1,\ldots, 4,\\
&C_5\oplus C_5.
\end{aligned}
\end{equation}
There exist infinitely many elliptic curves with each of the torsion groups from the list~\eqref{quart:list}, except for $C_{15}$ and $C_{16}$.
\end{tm}

\begin{proof}

As mentioned at the beginning of this section, we need only consider primes $p\leq 13$ as possible divisors of the order of a torsion point.

Before proceeding further, we find all elliptic curves over~$K$ containing a point of order~15 and show that the torsion subgroup of these curves is exactly $C_{15}$.  Recall that $X_1(15)$ is isomorphic to the elliptic curve with Cremona label 15a8. We compute
that the group of $K$-points of this elliptic curve is isomorphic to~$C_{16}$. Of the 16 points, 8 are cusps, and we compute that the remaining 8 points correspond to elliptic curves over~$K$ with torsion subgroup~$C_{15}$.  In particular, there exist no curves with torsion $C_{15n}$ for any integer $n>1$ and no curves with torsion $C_5 \oplus C_{15}$.

There exist elliptic curves with points of order $16$ over~$K$; one such example is the elliptic curve
$$y^2 + (6\zeta_5^3 + \zeta_5^2 + 3\zeta_5 + 6)xy + (-10\zeta_5^3 - 2\zeta_5^2 -
    5\zeta_5 - 8)y = x^3 + (-10\zeta_5^3 - 2\zeta_5^2 - 5\zeta_5 - 8)x^2,$$
on which $(0,0)$ is a point of order $16$. Since $X_1(16)$ has genus~2, the set $X_1(16)(K)$ is finite by Faltings's theorem.

It remains to prove that $E(K)$ does not have a subgroup isomorphic to one of the following groups:
$$
\begin{aligned}
&C_n & \text{where }n&=11,13,14,17,18,20,21,24,25,27,32,35,49,\\
&C_2\oplus C_{2m} & \text{where } m&=5,6,8,\\
&C_5\oplus C_{10}.
\end{aligned}
$$

\subsection{The cases $C_n$ for $n=11$, $14$, $20$, $49$ and $C_2\oplus C_{2m}$ for $m=5$, $6$}

These cases are dealt with by simply computing $X_1(n)(K)=X_1(n)(\Q)$, for $n=11$, $14$, $X_0(20)(K)=X_0(20)(\Q)$ for $C_{20}$ and $C_2\oplus C_{10}$, $X_0(24)(K)=X_0(24)(\Q)$ for $C_{24}$ and $C_2\oplus C_{12}$ as in the proof of Theorem~\ref{tm:cubic}. Similarly, noting that $X_0(49)$ is an elliptic curve, we compute $X_0(49)(K)=X_0(49)(\Q)$, which consists only of cusps.

\subsection{The case $C_{21}$}

The curve $X_0(21)$ is an elliptic curve and we compute that $X_0(21)(K)=X_0(21)(\Q)$. However, the difference between this case and the previous ones is that $Y_0(21)(K)$ is not empty and hence one needs also to check that the elliptic curves with 21-isogenies do not have any $K$-rational points over~$K$. This can be done by using division polynomials; see \cite{naj3,naj4} for details.

\subsection{The cases $C_{13}$ and $C_{18}$}

These cases are dealt with exactly as in the proof of Theorem~\ref{tm:cubic}.

\subsection{The cases $C_{27}$, $C_{32}$ and $C_2 \oplus C_{16}$}

We apply Theorem~\ref{thm:main} with $L=\Q$, $\fp_0=(11)$ (which is totally split in~$K$), $X'=X$ and $\pi=\id$.
The curves $X_1(27)$, $X_1(32)$ and $X_1(2,16)$ all have gonality $\ge5$ by \cite[Theorem~2.6]{jkp-quartic} (see also \cite{dvh}), and their Jacobians all have rank~0.
This implies conditions i), ii) and iii).
Condition iv) follows from the Hasse bound over $\F_{11}$, and condition v) clearly holds.
Finally, condition vi) holds in the case $C_{27}$ because $11$ is totally inert in $\Q(\zeta_{27})$, and in the cases $C_{32}$ and $C_2\oplus C_{16}$ because $S_{K,11}=\{1\}$ and there are no primes of degree~$1$ above~$11$ in the fields $\Q(\zeta_4)$, $\Q(\zeta_8)$, $\Q(\zeta_{16})^+$ and $\Q(\zeta_{32})^+$.

\subsection{The case $C_{25}$}

We apply Theorem~\ref{thm:main} with $L=\Q$, $\fp_0=(2)$ (which is totally inert in~$K$), $X'=X$ and $\pi=\id$.  The modular curve $X_1(25)$ has
gonality at least 5 \cite[Theorem~2.6]{jkp-quartic}, and $J_1(25)(\Q)$ has rank~0 \cite{Kubert} and trivial $2$-torsion; this implies conditions i), ii) and iii).
Although $25$ is in the Hasse interval of~$\F_{16}$, a search among all elliptic curves over~$\F_{16}$ shows that all such curves $E$ with 25 points satisfy $E(\F_{16})\simeq C_5\oplus C_5$. This shows that condition iv) holds.
Condition~v) clearly holds.
Finally, condition vi) is satisfied because $2$ is totally inert in $\Q(\zeta_{25})$.

\subsection{The case $C_{35}$}

We apply Theorem~\ref{thm:main} with $L=\Q$, $\fp_0=(3)$ (which is totally inert in~$K$), $X'=X$ and $\pi=\id$.
The curve $X_1(35)$ has gonality at least 5 \cite[Theorem~2.6]{jkp-quartic}, and $J_1(35)(\Q)$ has rank~0; this shows that conditions i) and ii) hold.
One easily checks conditions iii) and v).
By \cite[Theorem 4.1]{wat1}, there are no elliptic curves with order a multiple of 35 over $k(\fp)=\F_{81}$; this implies condition~iv).
The prime $3$ is totally inert in $\Q(\zeta_5)$, $\Q(\zeta_7)$ and $\Q(\zeta_{35})^+$, which implies condition vi).

\subsection{The case $C_{5} \oplus C_{10}$}

We apply Theorem~\ref{thm:main} with $L=\Q(\zeta_5)$, $\fp_0$ one of the primes above~$11$ (which is totally split in~$L$), $X'=\X_0(50)_L$ and $\pi$ the map defined by the inclusion $\alpha^{-1}\Gamma_1(5,10)\alpha \subset \Gamma_0(50)$, where $\alpha=\smallmat 5001$.
The gonality of~$X'$ is $2$, and we compute in Magma that $J_0(50)(K)$ has rank~$0$; this implies conditions i) and ii).
One easily checks conditions iii) and v).
Condition iv) follows from the Hasse bound over $k(\fp_0)=\F_{11}$.
Finally, condition vi) follows from the fact that all cusps of $X'$ are defined over~$L$.

\medbreak\noindent
Theorem~\ref{tm:quartic5} follows by combining the above cases.
\end{proof}

\section{Results for all cubic fields}

\label{sec:cubicgeneral}

We now apply the results of Section~\ref{sec:thm} to prove that certain groups are not in $\Phi(3)$. We note that the cases $C_{40}$, $C_{49}$ and $C_{55}$ (and more) were also proved independently by Wang \cite{wang}.

\begin{tm}
\label{tm:cubic-general}
The groups $C_2\oplus C_{20}$, $C_{40}$, $C_{49}$ and $C_{55}$ do not occur as subgroups of elliptic curves over cubic fields.
\end{tm}

\begin{proof}
To prove the cases $C_2\oplus C_{20}$, $C_{49}$ and $C_{55}$, we apply Corollary~\ref{cor} with $L=\Q$, $\fp_0=(3)$, $X'=X$ and $\pi=\id$.
We have $S_{\fp_0}=\{1,2,3\}$.
Conditions i) and ii) hold because $X_1(2,20)$, $X_1(49)$ and $X_1(55)$ have gonality at least 4 and their Jacobians have rank~0 over~$\Q$.
Condition iv) holds because $40$, $49$ and $55$ are all outside the Hasse intervals of $\F_3$, $\F_9$ and $\F_{27}$.
Condition iii) and v) clearly hold.
Condition vi) follows from the fact that there are no primes of degree 1, 2 or~3 above~3 in the fields $\Q(\zeta_n)$ for $n\in\{5,7,11\}$ and $\Q(\zeta_n)^+$ for $n\in\{20,49,55\}$.
Finally, the case $C_{40}$ follows from the case $C_2\oplus C_{20}$ in view of the covering $X_1(40)\to X_1(2,20)$.
\end{proof}

\section{Results for quartic fields}
\label{sec:quarticgeneral}

In this section, we show that certain groups of the form $C_m\oplus C_n$, with $m\mid n$ and $m\geq 3$, are not in $\Phi(4)$. Recall that an elliptic curve $E$ with a subgroup isomorphic to $C_m \oplus C_n$ has to be defined over a field containing $\Q(\zeta_m)$.

\begin{tm}
\label{tm:quartic-general}
The following groups do not occur as subgroups of elliptic curves over quartic fields:
$$
\begin{aligned}
C_3\oplus C_{12},\quad C_3\oplus C_{18},\quad C_3\oplus C_{27},\quad
C_3\oplus C_{33},\quad C_3\oplus C_{39},\quad C_4\oplus C_{12},\\
C_4\oplus C_{16},\quad C_4\oplus C_{28},\quad C_4\oplus C_{44},\quad
C_4\oplus C_{52},\quad C_4\oplus C_{68},\quad C_8\oplus C_{8\phantom{0}}.
\end{aligned}
$$
\end{tm}

\begin{proof}

We consider each of the above cases separately.



\subsection{The case $C_3\oplus C_{12}$}

The curve $X=X_1(3,12)$ has genus~3 and is non-hyperelliptic \cite{im}.  We apply Corollary~\ref{cor} with $L=\Q(\zeta_3)$, $\fp_0$ one of the primes of norm~$7$ in~$L$, $X'=X$ and $\pi=\id$.  We compute that the Jacobian of~$X$ has rank~0 over $\Q(\zeta_3)$ (see \cite[proof of Lemma~4.4]{jk-biel} for details).  This shows that conditions i), ii) and~iii) are satisfied.
For all elliptic curves over fields of 49 elements with 36 points, the group of points is isomorphic to $C_6\oplus C_6$, proving iv).
Condition~v) clearly holds.
Finally, condition vi) holds because $\fp_0$ is inert in $\Q(\zeta_{12})$.

\subsection{The case $C_3\oplus C_{18}$}

The curve $X=X_1(3,18)$ has genus~10, and its gonality over~$\C$ is at least 4 by the results of \cite{abr} and \cite[Appendix~2]{kim} (see also \cite[Theorem~1.2]{jkp-quartic}).
We use Corollary~\ref{cor}, choosing $L=\Q(\zeta_3)$, $\fp_0=(2)$, $X'=X$ and $\pi=\id$.
The Jacobian $J=J_1(3,18)$ over $\Q(\zeta_3)$ decomposes up to isogeny as
$$
J\sim \bigoplus_{i=1}^7 B_i,
$$
where $B_i$ is an elliptic curve for $1\le i\le 4$ and $B_i$ is an Abelian surface for $5\le i\le 7$.  A number of 2-descent and $L$-series computations in Magma shows that the rank of all these $B_i$ is 0.  This shows that condition ii) is satisfied.
If $\fq$ is one of the primes of norm~$7$ in~$L$, then $J(k(\fq))$ has order $3^{14}\cdot 7^3$.  This implies that the 2-torsion of $J(L)$ is trivial, so condition iii) holds.
The Hasse bound implies condition iv).  Condition v) clearly holds, and condition vi) holds because 2 is totally inert in $\Q(\zeta_{18})=\Q(\zeta_9)$.


\subsection{The case $C_3\oplus C_{27}$}

We use Corollary \ref{cor}, taking $L=\Q$, $\fp_0=(2)$, $X'=X_1(27)$ and $\pi\colon X\to X'$ the canonical map.
Condition i) holds because $X'$ has gonality~6 by \cite{dvh}.
A computation using $L$-functions shows that $\rk J'(\Q)=0$, which implies condition ii).
Condition iii) follows from $\#J'(\F_7)=242518973481$.
Condition iv) and~v) clearly hold.
Condition vi) holds since $2$ is inert in $\Q(\zeta_{27})$.




\subsection{The cases $C_3\oplus C_{33}$ and $C_3\oplus C_{39}$}

These cases require a slightly different approach, following the lines of~\cite{Kam2}.
Let $K$ be a quadratic extension of $\Q(\zeta_3)$, let $\sigma$ be the non-trivial element of $\Gal(K/\Q(\zeta_3))$, and let $\fp$ be a prime of~$K$ above~7.

We first decribe the case $C_3\oplus C_{33}$.
We take the hyperelliptic curve $X'=X_0(33)$ of genus 3, with hyperelliptic involution $w_{11}$. Suppose that $y$ is a non-cuspidal point on $X_1(3,33)$.
By Lemma \ref{lemma:mult}, $y$ maps to the cusp at $\infty$ mod~$\fp$, and $y^\sigma$ maps to $\infty$ mod~$\fp$.
The points $y$ and $y^\sigma$ on $X_1(3,33)$ map to $x$ and $x^\sigma$ on~$X'$, which likewise map to $\infty$ mod~$\fp$.
Consider the map
$$
\begin{aligned}
f\colon X'&\rightarrow J'\\
t&\mapsto[t+t^\sigma-2\infty].
\end{aligned}
$$
Then $f(x)$ is $\Q(\zeta_3)$-rational, and $f(x)$ mod~$\fp$ is $0$.
We compute that $J'(\Q(\zeta_3))$ is finite.
Since reduction modulo~$\fp$ is injective on the torsion, it follows that $f(x)=0$ over $\Q(\zeta_3)$, so there is a function~$g$ whose divisor is $x+x^\sigma-2\infty$.
Since $g$ has degree~2, it is fixed by the hyperelliptic involution.
It follows that $\infty$ is fixed by the hyperelliptic involution.
But $w_{11}$ acts freely on the cusps of $X'$, leading to a contradiction.

We now deal with the case $C_3\oplus C_{39}$.  We take
$$
X'=X_0(39)/w_{13}\colon y^2 = x^6 - 20x^4 - 6x^3 + 64x^2 - 48x + 9.
$$
The curve $X'$ is hyperelliptic of genus 2, and the hyperelliptic involution on $X'$ is induced by~$w_3$.
We compute that $J'(\Q(\zeta_3))$ is finite.
Using the same arguments as above, we conclude that $w_3$ fixes the cusp at $\infty$ of $X'$, but $w_3$ acts by switching the two cusps $0$ and $\infty$, which leads to a contradiction.

\subsection{The case $C_4\oplus C_{12}$}

We apply Corollary~\ref{cor} with $L=\Q(i)$, $\fp_0=(2+i)$, $X'=X=X_1(4,12)$ and $\pi=\id$.
The curve $X$ has genus~5 and is non-hyperelliptic \cite{im}.
It is the base change of the curve $X_{\Delta'}(48)$ over~$\Q$, where $\Delta'$ is the subgroup $\{\pm1,\pm13,\pm25,\pm37\}$ of $(\Z/48\Z)^\times$.
By \cite[pages 464--465]{jk-biel}, the Jacobian $J_{\Delta'}$ of~$X_{\Delta'}$ decomposes over~$\Q$ as
$$
J_{\Delta'}\sim B_1^2\oplus B_2\oplus B_3,
$$
where
$$
\begin{aligned}
B_1\colon y^2&=x^3-x^2-4x+4,\\
B_2\colon y^2&=x^3+x^2-4x-4
\end{aligned}
$$
and $B_3$ is the Jacobian of the curve
$$
C_3\colon y^2=x^5-10x^3+9x.
$$
Note that $B_2$ is a $-1$-twist of $B_1$, and hence $B_1$ and $B_2$ are isomorphic over $\Q(i)$.
A computation in Magma shows that
$$\rk B_1(\Q(i))=\rk B_2(\Q(i))= \rk B_3(\Q(i))=0,$$
so condition ii) is satisfied.
Condition iii) holds trivially.
The Hasse bound implies condition iv).
Condition v) is clearly satisfied.
Finally, condition vi) holds because $\fp_0$ is totally inert in $\Q(\zeta_{12})$.


\subsection{The case $C_4\oplus C_{16}$}

We take the genus 5 curve $X'=X_1(2,16)$ of gonality $4$ \cite[Theorem 2.8.]{jkp-quartic}, $L=\Q(i)$ and $\fp_0=(2+i)\Z[i]$. The Jacobian $J'$ of $X'$ factors over~$\Q$ as
$$J'\sim E_1 \oplus J_1(16)^2,$$
where $E_1$ is the elliptic curve over $\Q$ with Cremona label 32a1.
A computation using $2$-descent shows that $\rk J'(\Q(i))=0$ and $\fp_0$ is inert in $\Q(\zeta_{16})$. As all the assumptions of Corollary $\ref{cor}$ are satisfied, the result follows.

\subsection{The case $C_4\oplus C_{28}$}

We take $X'=X_1(28)$ of gonality 6 \cite{dvh}, $L=\Q$ and $\fp_0=(3)$.
We compute $\rk J'(\Q)=0$, showing that condition ii) holds.
Condition vi) is satisfied because $S_{K,\fp_0}=\{1,2,4\}$, while 3 is inert in~$\Q(i)$ and none of the fields $\Q(\zeta_7)$, $\Q(\zeta_{14})^+$ and $\Q(\zeta_{28})^+$ have any primes of degree 1, 2 or~4 above~3.

\subsection{The case $C_4\oplus C_{44}$}

We take $X'=X_1(44)$ of gonality $\geq 7$ \cite{abr}, $L=\Q$ and $\fp_0=(3)$.
Condition vi) is satisfied because $S_{K,\fp_0}=\{1,2,4\}$, while 3 is inert in $\Q(i)$ and none of the fields $\Q(\zeta_{11})$, $\Q(\zeta_{22})^+$ and $\Q(\zeta_{44})^+$ have any primes of degree 1, 2 or~4 above~3.
After computing $\rk J'(\Q)=0$, we are done.   

\subsection{The case $C_4\oplus C_{52}$}

We take $X'=X_1(26)$ of gonality 6 \cite{dvh}, $L=\Q$ and $\fp_0=(3)$. We check that 3 splits into 4 primes of degree~3 in $\Q(\zeta_{13})$, all of which remain inert in $\Q(\zeta_{52})$.  As $(3)$ is inert in $\Q(i)$, the primes of any quartic field containing $\Q(i)$ above it are of degree $2$ or $4$, this proves vi). We compute $\rk J'(\Q)=0$ and $\#J'(\F_5)=3^2 5^2 7^3 19^2$, proving that all the assumptions are satisfied. 

\subsection{The case $C_4\oplus C_{68}$}

We use Corollary \ref{cor}, taking $X'=X_1(34)$ of gonality~10 \cite{dvh}, $L=\Q$ and $\fp_0=(3)$. Since $3$ splits into 2 primes of degree~$16$ over $\Q(\zeta_{68})$ and is inert in $\Q(\zeta_{17})$ and $\Q(i)$, it easily follows that condition vi) is satisfied. We compute $\rk J'(\Q)=0$, completing the proof.




\subsection{The case $C_8\oplus C_8$}

The only field over which there could exist an elliptic curve with full $8$-torsion is $\Q(\zeta_8)=\Q(i, \sqrt 2)$.  To show that such curves do not exist, we will in fact prove the stronger statement that there does not exist an elliptic curve over $\Q(\zeta_8)$ with a subgroup isomorphic to $C_4 \oplus C_8$.
To prove this, we note that the modular curve $X_1(4,8)$ is isomorphic (over $\Q(i)$) to the elliptic curve with Cremona label 32a2 \cite[Lemma~13]{naj}. We compute
$$X_1(4,8)(\Q(\zeta_8))\simeq C_4 \oplus C_4,$$
and all the points are cusps, which proves our claim.

\medbreak\noindent
This finishes the proof of Theorem~\ref{tm:quartic-general}.
\end{proof}

\begin{acknowledgments}
We would like to thank Maarten Derickx for useful discussions on the topic of this paper.
\end{acknowledgments}


\begin{thebibliography}{99}
\bibitem{abr}
D. Abramovich, \emph{A linear lower bound on the gonality of modular curves}, Int. Math. Res. Notices \textbf{20} (1996), 1005--1011.

\bibitem{blr}
S. Bosch, W. L\"utkebohmert and M. Raynaud, N\'eron Models.  Springer-Verlag, Berlin/\allowbreak Heidel\-berg, 1990.

\bibitem{bbdn}
J. G. Bosman, P. J. Bruin, A. Dujella and F. Najman, \emph{Ranks of elliptic curves with prescribed torsion over number fields}, Int. Math. Res. Notices \textbf{2014} (2014), 2885--2923.

\bibitem{derickx}
M. Derickx, \emph{Torsion points on elliptic curves and gonalities of modular curves}, master's thesis, Universiteit Leiden, 2012, \url{http://www.math.leidenuniv.nl/nl/theses/324/}

\bibitem{maz_new}
M. Derickx, S. Kamienny and B. Mazur, \textit{Rational families of 17-torsion points of elliptic curves over number fields}, preprint, \url{http://www.math.harvard.edu/~mazur/papers/For.Momose20.pdf}

\bibitem{dkss}
M. Derickx, S. Kamienny, W. A. Stein and M. Stoll, \textit{Torsion points on elliptic curves over fields of small degree}, preprint.

\bibitem{dvh}
M. Derickx and M. van Hoeij, \emph{Gonality of the modular curve $X_1(N)$}, J. Algebra \textbf{417} (2014), 52--71.

\bibitem{im}
N. Ishii and F. Momose, \emph{Hyperelliptic modular curves}, Tsukuba J. Math. \textbf{15} (1991), 413--423.

\bibitem{jk-biel}
D. Jeon and C.~H.~Kim, \emph{Bielliptic modular curves $X_1(m,n)$}, manuscripta math. \textbf{118} (2005), 455--466.

\bibitem{jk-delta}
D. Jeon and C.~H.~Kim, \emph{On the arithmetic of certain modular curves}, Acta Arith. \textbf{130} (2007), 181--193.

\bibitem{jks}
D. Jeon, C.~H.~Kim and A. Schweizer, \emph{On the torsion of elliptic curves over cubic number fields}, Acta Arith. \textbf{113} (2004) 291--301.

\bibitem{jkp-quartic}
D. Jeon, C.~H.~Kim and E. Park, \emph{On the torsion of elliptic curves over quartic number fields}, J. London Math. Soc. \textbf{74} (2006), 1--12.

\bibitem{Kam3}
S. Kamienny, \textit{Torsion points on elliptic curves over all quadratic fields}, Duke Math J. \textbf{53} (1986), 157--162.

\bibitem{Kam2}
S. Kamienny, \emph{Torsion points on elliptic curves over all quadratic fields. II.}, Bull. Soc. Math. France \textbf{114} (1986), 119--122. 

\bibitem{Kam1}
S. Kamienny, \emph{Torsion points on elliptic curves and $q$-coefficients of modular forms}, Invent. Math. \textbf{109} (1992), 221--229.

\bibitem{kn}
S. Kamienny and F. Najman, \emph{Torsion groups of elliptic curves over quadratic fields}, Acta. Arith. \textbf{152} (2012), 291--305.

\bibitem{kato}
K. Kato, \textit{$p$-adic Hodge theory and values of zeta functions of modular forms}, Cohomologies $p$-adiques et applications arithm\'etiques. III.  Astérisque \textbf{295} (2004), 117--290. 

\bibitem{kat}
N.~M.~Katz, \emph{Galois properties of torsion points on abelian varieties}, Invent. Math. \textbf{62} (1981), 481--502.

\bibitem{km}
M.~A.\ Kenku and F.\ Momose, \emph{Torsion points on elliptic curves defined over quadratic fields}, Nagoya Math.\ J.\ \textbf{109} (1988), 125--149.

\bibitem{kim}
H.~H.~Kim, \emph{Functoriality for the exterior square of $\GL_4$ and the symmetric fourth of $\GL_2$}, J. Amer. Math. Soc. \textbf{16} (2003), 139--184, With appendix 1 by D. Ramakrishnan and appendix
2 by H.~H.~Kim and P. Sarnak.

\bibitem{Kubert}
D. S. Kubert, \emph{Universal bounds on the torsion of elliptic curves}, Proc. London Math. Soc. \textbf{33} (1976), 193--237.

\bibitem{mag}
W.\ Bosma, J.~J.\ Cannon, C.\ Fieker, A.\ Steel (eds.), Handbook of Magma functions, Edition 2.19 (2013).

\bibitem{maz1}
B. Mazur, \emph{Modular curves and the Eisenstein ideal},  Inst. Hautes \'Etudes Sci. Publ. Math. \textbf{47} (1978),  33--186.

\bibitem{maz2}
B. Mazur, \emph{Rational isogenies of prime degree}, Invent. Math. \textbf{44} (1978), 129--162.

\bibitem{mt}
B. Mazur and J. Tate, \emph{Points of order 13 on elliptic curves}, Invent. Math. \textbf{22} (1973), 41--49.

\bibitem{naj0}
F. Najman, \emph{Complete classification of torsion of elliptic curves over quadratic cyclotomic fields}, J. Number Theory \textbf{130} (2010), 1964--1968.

\bibitem{naj}
F. Najman, \emph{Exceptional elliptic curves over quartic fields}, Int. J. Number Theory \textbf{8} (2012), 1231--1246

\bibitem{naj3}
F. Najman, \emph{Torsion of elliptic curves over cubic fields}, J. Number Theory, \textbf{132} (2012), 26--36.

\bibitem{naj4}
F. Najman, \emph{Torsion of rational elliptic curves over cubic fields and sporadic points on $X_1(n)$}, Math. Res. Letters, to appear.

\bibitem{ogg}
A.~P.~Ogg, \emph{Hyperelliptic Modular Curves}, Bull. Soc. Math. France, \textbf{102} (1974), 449--462.

\bibitem{Par1}
P. Parent, \emph{Torsion des courbes elliptiques sur les corps cubiques}, Ann. Inst. Fourier \textbf{50} (2000), 723--749.

\bibitem{Par2}
P. Parent, \emph{No 17-torsion on elliptic curves over cubic number fields}, J. Theor. Nombres Bordeaux \textbf{15} (2003), 831--838.

\bibitem{sil}
J. H. Silverman, Arithmetic of elliptic curves, 2nd edition, Springer-Verlag, New York, 2009.

\bibitem{wang}
J. Wang, \textit{On the cyclic torsion of elliptic curves over cubic number fields}, preprint, \url{http://arxiv.org/abs/1502.06873}

\bibitem{wat1}
W. Waterhouse, \textit{Abelian varieties over finite fields}, Ann. Sci. \' Ecole Norm. Sup. \textbf{4} (1969), 521--560.

\end{thebibliography}
\end{document}